\documentclass[10pt,reqno]{amsart}

\usepackage{amsmath}
\usepackage{amsthm}
\usepackage{amsfonts}
\usepackage{amssymb,latexsym}
\usepackage[all]{xy}
\usepackage{graphics}
\usepackage{graphicx}
\usepackage[mathscr]{eucal}
\usepackage{verbatim}
\makeatletter
\def\HyPsd@CatcodeWarning#1{}
\makeatother

\theoremstyle{plain}
    \newtheorem{thm}{Theorem}[section]
    
    \newtheorem{lemma}[thm]{Lemma}

\theoremstyle{definition}

\theoremstyle{remark}
    \newtheorem{rem}[thm]{Remark}

\newcommand{\rar}{\ensuremath{\rightarrow}}

\newcommand{\la}{\langle}
\newcommand{\ra}{\rangle}

\renewcommand{\bar}{\mathbf}

\newcommand{\Gal}{\text{Gal}}

\begin{document}

\title{Witt's Cancellation theorem seen as a cancellation}

\dedicatory{Dedicated to the late Professor Amit Roy}

\author{Sunil K. Chebolu}
\address{Department of Mathematics \\
Illinois State University\\
Normal, IL 61761, USA}
\email{schebol@ilstu.edu}

\author{Dan {Mc}Quillan}
\address{Department of Mathematics \\
Norwich University \\
Northfield, VT 05663, USA}
\email{dmcquill@norwich.edu}

\author{J\'{a}n Min\'{a}\v{c}}
\address{Department of Mathematics\\
University of Western Ontario\\
London, ON N6A 5B7, Canada}
\email{minac@uwo.ca}



\begin{abstract}

The year 2017 marks the 80th anniversary of Witt's famous paper containing key results, including the Witt cancellation theorem, which form  the foundation for the algebraic theory of quadratic forms. We pay homage to this paper by presenting a  transparent and algebraic proof of the Witt cancellation theorem, which itself is based on a cancellation. We also present an overview of some recent spectacular work  which is still building on Witt's original  creation of the algebraic theory of quadratic forms.
\end{abstract}

\maketitle
\thispagestyle{empty}


\section{Introduction}

   The algebraic theory of quadratic forms will soon celebrate its 80th  birthday. Indeed, it was 1937 when Witt's pioneering paper \cite{Witt} -- a mere 14 pages -- first introduced many beautiful  results  which form the foundation for the algebraic theory of quadratic forms. In particular, Witt describes the construction of the Witt ring itself and the essential fact needed for its construction over an arbitrary field. This result, originally Satz 4 in \cite{Witt}, is now formulated as the Witt cancellation theorem, and it is the technical heart of Witt's brilliant idea to study the collection of \emph{all} quadratic forms over a given field as a single algebraic entity. Prior to Witt's paper, quadratic forms were studied one at a time. However Witt showed that  a certain collection of quadratic forms under an equivalence relation can be equipped with the structure of a commutative ring. Indeed, Satz 6 says:
\begin{center}
 ``\emph{Die Klassen \"{a}hnlicher Formen bilden einen Ring}"
\end{center}
which means, ``The classes of similar forms, form a ring."   In order to honor Witt's  contributions, this ring is now called the \emph{Witt ring}.

 The Witt ring remains  a central object of study, even 80 years after its birth. Building on Voevodsky's Fields medal winning work from 2002, Orlov, Vishik and Voevodsky settled Milnor's conjecture \cite{OVV} on quadratic forms, which is a deep statement about the structure of the Witt ring. This  work uses sophisticated tools from algebraic geometry and homotopy theory to  provide a complete set of invariants for quadratic forms, extending the classical invariants known to Witt \cite{Witt}, including dimension, discriminant and the Clifford invariant.

In addition to its crucial role in defining the Witt ring, the Witt cancellation theorem also has other important applications, such as establishing Sylvester's law of Inertia, which classifies quadratic forms over the field of real numbers. Clearly the Witt cancellation theorem  is special and therefore deserves further analysis. The main goal of this paper is to present a  transparent and algebraic proof to complement the classical geometric proof, and then carefully compare the two approaches.

       The paper is organized as follows. In Section 2 we state the Witt cancellation theorem,  guide the reader towards our proof of the cancellation theorem, and then present the proof itself. A geometric approach to Witt cancellation,  based on hyperplane reflections, is presented in Section 3. In Section 4 we provide a comparison between the algebraic and geometric approaches. Using Witt cancellation as the key, we review  the construction of the   Witt ring of quadratic forms in Section 5. In Section 6  we present an informal overview of  the  Milnor conjectures on quadratic forms and some dramatic recent  developments  culminating in the proof of the Bloch-Kato conjecture, which is a considerable extension of part of the original Milnor conjectures. In Section 7 we reveal an interesting surprise.   All sections, except  possibly Section 6, can be read profitably  by any reader who is familiar with basic linear algebra.

We begin with some preliminaries. Throughout the paper we assume that our base field $F$ has characteristic not equal to $2$.  There are several equivalent definitions of a quadratic form. The following is probably the most commonly used definition. An $n$-ary \emph{quadratic form} $q$ over $F$ is a homogeneous polynomial of degree $2$ in $n$ variable over $F$:
\[ q = \sum_{i, j = 1}^n a_{ij} x_i x_j  \ \ \text{ for } \ \ a_{ij} \text{ in } F. \]
It is customary to render the coefficients symmetric by writing
\[ q = \sum_{i, j = 1}^n b_{ij} x_i x_j, \ \ \text{ where } b_{ij} = \frac{a_{ij} + a_{ji}}{2},\]
therefore $b_{ij} = b_{ji}.$ (This is possible because the characteristic of our field is not $2$.)

If we view $\bar{x} = (x_{1}, x_{2}, \dots, x_{n})$ as a column vector, and its transpose $\bar{x}^t$ as a row vector, then we can write
\[ q (\bar{x}) = \bar{x}^{t} B \bar{x}, \]
where $B = (b_{ij})$ is an $n \times n $ matrix.
In other words, we associate $q$ with a symmetric matrix $B$ which also defines a symmetric bilinear form on $V \times V$, where $V=F^n$.
Two $n$-ary quadratic forms $q_{a}$ and $q_{b}$ are \emph{equivalent}, or \emph{isometric} if for some non-singular $n \times n$ matrix $M$
we have
\[  q_{a}(\bar{x}) = q_{b}(M\bar{x}). \]
In this case we write $q_{a} \cong q_{b}$.  Recall that two symmetric matrices $A$ and $B$ are said to be \emph{congruent} if there exist and invertible matrix $M$ such that $A = M^tBM$.  Equivalence classes of quadratic forms thus correspond to congruence classes of symmetric matrices \footnote{To illustrate the notion of equivalence of quadratic forms, consider the quadratic form $q_b(z_1, z_2) = 5z_1^2 - 2z_1 z_2 + 5z_2^2$ over the field $\mathbb{R}$ of real numbers.  What is the conic section that is  represented by the equation $q_b(z_1, z_2) = 1$?  The given quadratic form is equivalent over $\mathbb{R}$  to the  form $q_a(x_1, x_2) = 4x_1^2 + 6x_2^2$. The equivalence is given by the equations
\begin{eqnarray*}
z_1 & = &  \frac{1}{\sqrt{2}}  x_1 - \frac{1}{\sqrt{2}} x_2, \text{ and }\\
z_2 & = &  \frac{1}{\sqrt{2}} x_1 + \frac{1}{\sqrt{2}}  x_2.
\end{eqnarray*}
It is clear that the new equation $4x_1^2 + 6x_2^2 =1$ represents an ellipse, and therefore so does the original equation. }.

The following useful result is well known and can be found in any standard textbook on quadratic forms; see for example \cite{Lam} or \cite{Scharlau}.
\begin{thm} An $n$-ary quadratic form over a field $F$ of characteristic not equal to $2$ is equivalent to a diagonal form, i.e., a form that is equal to $a_{1} x_{1}^{2} + \cdots +  a_{n}x_{n}^{2}$ for
some field elements $a_{1}, \dots, a_{n}$.
\end{thm}
For brevity we shall denote the diagonal quadratic form $a_{1} x_{1}^{2} + \cdots +  a_{n}x_{n}^{2}$
by $\la a_{1}, \dots , a_{n} \ra$.
In view of this theorem it is enough to study diagonal forms over $F$. Furthermore, we assume that our diagonal quadratic forms are \emph{non-degenerate}, i.e., $a_i \ne 0$ for $i = 1, \dots, n$. The number $n$ is called the dimension of $q$.

\section{Witt Cancellation: algebraic approach} \label{Witt}

In this section we will  present a  transparent and algebraic proof
of the Witt cancellation theorem to complement the classical geometric proof.
The following is the simplest form of the Witt cancellation theorem.  Other general statements can be easily derived from this simple form.

\begin{thm} \textup{(Witt cancellation)} Let $q_a = \la a_1, a_2, \dots, a_n \ra$ and $q_b = \la b_1, b_2, \dots, b_n\ra$ be non-degenerate $n$-ary quadratic forms over a field $F$ of characteristic not equal to $2$, with
$n > 1$, and assume that $a_1 = b_1$. If there is an isometry $q_a \cong q_b$, then there is another isometry $\la a_2, \dots, a_n \ra \cong \la b_2, \dots, b_n\ra$.
\end{thm}

Before presenting our proof, we will explain the key idea in such a way that the reader may build the
proof before even reading it -- a \emph{guided self-discovery} approach.  Witt's cancellation theorem essentially says that we may ``cancel" a
common term, $a_1$, from both sides of a given isometry, in order to obtain a new isometry. We want a
proof that reflects this cancellation directly.  To this end, recall that by the definition of isometry, there is an invertible linear
transformation
\begin{equation} \label{key}
 z_{i} = m_{i1}x_{1} + \cdots + m_{in}x_{n}, \ \ i=1, \dots , n, \ \ m_{ir} \in  F,
\end{equation}
which takes $q_{b}$ to  $q_{a}$.
This means that the isometry
\begin{equation} \label{key2}
a_1 x_1^2 + a_2 x_2^2 + \cdots + a_n x_n^2  \cong b_1 z_1^2 + b_2 z_2^2 + \cdots + b_n z_n^2
\end{equation}
becomes a polynomial identity in the $n$ variables $x_1, x_2, \dots, x_n$ after using the $n$ transformations in Equation (\ref{key}).
Our idea then is to simply take this one step further, by substituting $x_1$ with a carefully chosen linear
combination of the remaining $n-1$ variables $x_2, \dots, x_n$ so that in Equation (\ref{key2}), the first term on the left hand side  will cancel with the
first term on the right hand side. This will then give us our desired isometry.
So we now ask: what is this magical substitution? In other words, which linear combination do we use for $x_1$?  If $x$ is the answer to this question, then it should satisfy the equation
\[a_1x^2   = b_1(mx+y)^2,\ \ \text{ where we set }  m:=m_{11} \ \ \text{ and } \ \ y := m_{12}x_2 + \cdots + m_{1n}x_n.\]
However, since $a_1 = b_1$, it is sufficient that our $x$ satisfy
\begin{equation}
x = mx+ y.
\end{equation}
Note that this last equation  reminds one of the usual cancellation from high school where we learned
 how to solve linear equations:
\[x = m x + y  \ \ \ \implies \ \ \ x = \frac{y}{1 - m} \ \ \text{if } m \ne 1.\]

After this motivational warm-up, it is now time to give a formal proof. The reader will see that our proof will be quite transparent and will be based on the simple identity:

\begin{equation} \label{eq0}
\boxed{ \frac{y}{1 - m} = \frac{my}{1 - m} + y}
\end{equation}

\noindent
\emph{Proof (of the Witt cancellation theorem).} Since $q_a \cong q_b$, we can write
\begin{equation} \label{eq1}
 a_1 x_1^2 + \cdots + a_n x_n^2 = q_a(\bar{x}) = q_b(M\bar{x}) = b_1 z_1^2 +
\cdots + b_n z_n^2,
\end{equation}
where $M = (m_{ij})$ is an $n \times n $ invertible matrix over $F$ and $z_i = m_{i1}x_1
+ \cdots + m_{in}x_n$ for $i$ from $1$ to $n$.
We first argue that $m_{11}$ can be assumed without loss of generality to be not equal to $1$.
As a matter of fact,  if $m_{11} =1$, then we replace $m_{1k}$ with $-m_{1k}$ for all $k$.
This changes $z_1$ to $-z_1$. However, that does not effect Equation (\ref{eq1}).
So we  assume without loss of generality that $m_{11} \ne 1$.

To prove our theorem we would like to cancel the first terms ($a_1 x_1^2$ and $b_1 z_1^2$) on
either sides of Equation (\ref{eq1}). To do this, in Equation \ref{eq1} we make the substitution
\begin{equation} \label{substitution}  x_1 = \frac{y}{1 - m_{11}},\end{equation}
where
\begin{equation} \label{eq2}
 y := z_1 - m_{11}x_1 = m_{12}x_2 + \cdots + m_{1n} x_{n} .
\end{equation}
 Note that this is a valid substitution because $m_{11} \ne 1$. Moreover, this substitution expresses $x_1$  as a linear combination of $x_2, \dots, x_n$.
This substitution, in conjunction with the  assumption $a_1 = b_1$ and our identity (\ref{eq0}),
gives the following equations.
\begin{eqnarray*}
a_1 x_1^2 & = & a_1\left( \frac{y}{1 - m_{11}}\right)^2 \\
           & =  & b_1\left( \frac{y}{1 - m_{11}}\right)^2    \\
					& = &   b_1\left( y + m_{11}\left(\frac{y}{1 - m_{11}}\right)\right)^2  \ \ \  \text{ from identity }  (\ref{eq0}) \\
					& =  & b_1 (y + m_{11}x_1)^2      \\
					& =  & b_1 z_1^2
\end{eqnarray*}
Therefore we can cancel these two terms in our original equation (\ref{eq1}), which now reduces to one in $2(n-1)$ variables:
\begin{equation}
 a_2 x_2^2 + \cdots + a_n x_n^2  = b_2 z_2^2 +
\cdots + b_n z_n^2.
\end{equation}
In this new equation, for $i \ge 2$, $z_i$ is expressed as a linear combination of $x_2, x_3, \dots, x_n$, say $ z_i = w_i(x_2, x_3, \dots, x_n)$. It remains to show that
this linear transformation is invertible.  To see this, let $N = (n_{ij})$ be the change-of-coordinates matrix which corresponds to our linear transformation
$ z_i = w_i(x_2, x_3, \dots, x_n),  i= 2, \dots, n$. Then the transformation between the $(n-1)$-ary forms $s_{a} := \la a_2, \dots, a_n \ra$ and
$s_{b} := \la b_2, \dots,  b_n\ra$ is given by the matrix equation
\[ A = N^{T} B N,\]
where $A$ and $B$ are the diagonal matrices representing the forms $s_{a}$
and $s_{b}$ respectively. Taking determinants on both sides of the last equation, we get
\[ \det(A) = \det(B) (\det(N))^{2}.\]
Since $s_a$ is non-degenerate, $\det(A)$  is non-zero and therefore $\det(N)$ is also non-zero. This shows that  $N$ is invertible.
Thus we have shown that the forms $s_a$ and $s_b$ are isometric.  \hspace{1 mm}$\Box$

\begin{rem} In the above proof, we see that the Witt cancellation theorem actually follows from the formal algebraic cancellation of like terms in a polynomial identity, explaining the title of our paper.
\end{rem}

\section{Witt Cancellation:  geometric approach} \label{geometric}

In this section we present the standard, coordinate-free, geometric approach to quadratic forms and the Witt cancellation theorem.

 A \emph{quadratic space} is a finite-dimensional $F$-vector space equipped with a symmetric bilinear form
\[ B \colon V \times V \rar F. \]
The associated quadratic form $q \colon V \rar F$ is obtained by setting $q (\bar{v}) = B(\bar{v}, \bar{v})$.  The bilinear form $B$ can be recovered from $q$
because of  the identity
\[B(\bar{x},\bar{ y}) = \frac{1}{2}(q(\bar{x}+ \bar{y}) - q(\bar{x}) - q(\bar{y})),\]
as one can easily check.  Therefore a  quadratic space can be denoted by  $(V, B)$, or equivalently  by $(V, q)$.

Coordinate free definitions in quadratic form theory are naturally analogous to their coordinate counterparts.
For instance,  an \emph{isometry}  between $(V,  B_1)$ and $(V, B_2)$ is a linear
isomorphism $T \colon V \rightarrow V$ such that $B_2(\bar{x}, \bar{y}) = B_1(T(\bar{x}), T(\bar{y}))$ for all $\bar{x}$ and $\bar{y}$ in $V$.
Vectors $\bar{x}$ and $\bar{y}$ in $V$ are said to be \emph{orthogonal} if $B(\bar{x}, \bar{y}) = 0$. A quadratic space $(V, B)$ is \emph{non-degenerate} if $B(\bar{v}, \bar{w}) = 0$ for all $\bar{w}$ in $V$ implies $\bar{v} = \bar{0}$.
Given two quadratic spaces $(V_1, q_1)$ and $(V_2, q_2)$, there is a natural quadratic from on the space $V_1 \oplus V_2$ which is defined by
\[ q( (\bar{x}_1, \bar{x}_2)) :=  q_{1}(\bar{x}_1) +  q_{2}(\bar{x}_2). \]
This quadratic space is denoted by $(V_1, q_1) \, \, \bot \, \, (V_2, q_2).$

The geometric form of  the Witt cancellation theorem in its simplest form can now be stated as follows.

\begin{thm} \label{WittCancellation-geometric}
Let  $(V, q)$ be an $n$-dimensional non-degenerate quadratic form with $n >1$, and
let  $\{ \bar{e}_1,  \dots, \bar{e}_n \}$ and  $\{ \bar{f}_1, \dots, \bar{f}_n \}$ be
 two orthogonal bases for $(V, q)$.
 If $q(\bar{e}_1) = q(\bar{f}_1)$, then $q$ restricted to $\text{Span}\{\bar{e}_2,  \dots, \bar{e}_n\}$  is isometric to $q$ restricted to
 $\text{Span} \{ \bar{f}_2,  \dots, \bar{f}_n \}$.
\end{thm}

Given a quadratic space $(V, q)$ and a vector $\bar{u}$  in $V$ such that $q(\bar{u}) \ne 0$,  the map
\[ \tau_{\bar{u}} (\bar{z}) := \bar{z} - \frac{2B(\bar{z}, \bar{u})}{q(\bar{u})} \bar{u}\]
can be easily shown to be an isometry of $(V, q)$; see \cite[Page 13]{Lam}.  In fact, this map is  the reflection in the plane perpendicular to $\bar{u}$. A key ingredient in the proof of  Theorem \ref{WittCancellation-geometric} is the following  hyperplane reflection lemma.
\begin{center}
\begin{figure}[!h]
\includegraphics[height=65mm, width=40mm]{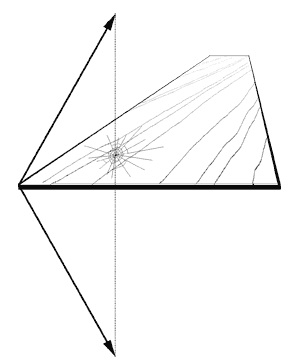}
\caption{Hyperplane reflection.}
\end{figure}
\end{center}
\begin{lemma}
Let $(V, q)$ be a quadratic space and let $\bar{x}$ and $\bar{y}$ be two vectors in $V$ such that $q(\bar{x}) = q(\bar{y}) \ne 0$. Then there exists an
isometry $\rho \colon (V, q) \cong (V , q)$ which sends $\bar{x}$ to $\bar{y}$.
\end{lemma}

\begin{proof}
Note that
\[ q(\bar{x} + \bar{y}) + q(\bar{x} - \bar{y})  = B(\bar{x}+\bar{y}, \bar{x}+\bar{y}) + B(\bar{x} - \bar{y}, \bar{x}-\bar{y}) = 2q(\bar{x}) + 2q(-\bar{y}) = 4 q(\bar{x}) \ne 0.\]
This means $q(\bar{x}+\bar{y})$ and $q(\bar{x}-\bar{y})$ both cannot be zero simultaneously.
 Suppose that  $q(\bar{x} - \bar{y}) \ne 0$.  Then $\tau_{\bar{x} - \bar{y}}$ is an isometry that maps $\bar{x}$ to $\bar{y}$.
To see this, first note that
\[q(\bar{x} - \bar{y}) = B(\bar{x} , \bar{x}) + B(\bar{y}, \bar{y}) - 2B(\bar{x} , \bar{y}) = 2B(\bar{x} , \bar{x}) - 2B(\bar{x}, \bar{y}) = 2B(\bar{x}, \bar{x}-\bar{y}).\]
Therefore,
\[ \tau_{\bar{x} -\bar{y}} (\bar{x}) = \bar{x} - \frac{2B(\bar{x}, \bar{x}-\bar{y})}{q(\bar{x}-\bar{y})}(\bar{x} -\bar{y}) = \bar{x} -(\bar{x}-\bar{y}) = \bar{y}.\]
If $q(\bar{x}+\bar{y}) \ne 0$, then since $q(\bar{x} + \bar{y}) = q(\bar{x} -(-\bar{y}))$, the above argument shows that $\tau_{\bar{x} + \bar{y}} (\bar{x}) = \tau_{\bar{x} - (-\bar{y})} \bar{x} = -\bar{y}$, and therefore
$-(\tau_{\bar{x}+\bar{y}} \bar{x}) = \bar{y}$. This completes the proof our lemma.
\end{proof}

\vskip 3mm
\noindent
\textbf{Proof of Theorem 3.1} We are given that $q(\bar{e}_1) = q(\bar{f}_1)$. This common value cannot be zero because $q$ is non-degenerate. Therefore, as
 observed in the proof of the above lemma  $q(\bar{e}_1+\bar{f}_1)$ and $q(\bar{e}_1-\bar{f}_1)$ both  cannot be zero simultaneously.
 Replacing $\bar{f}_1$ with $-\bar{f}_1$ if necessary, we may assume that $q(\bar{e}_1-\bar{f}_1) \ne 0$. Then we claim that the isometry
\[ \tau_{\bar{e}_1-\bar{f}_1}\]
does the job. That is, it gives an isometry between $\bar{e}_1^{\perp} :=\text{Span}\{\bar{e}_2,  \dots, \bar{e}_n\}$ and \break $\bar{f}_1^{\perp} := \text{Span} \{ \bar{f}_2, \dots, \bar{f}_n \}$.  Indeed, from the above lemma, the map  $ \tau_{\bar{e}_1-\bar{f}_1}$
takes $\bar{e}_1$ to $\bar{f}_1$. Since  $\tau_{\bar{e}_1-\bar{f}_1}$ is an isometry of $(V, B)$, it maps $\bar{e}_1^{\perp}$ to $\bar{f}_1^{\perp}$. Thus $ \tau_{\bar{e}_1-\bar{f}_1}$ restricts to a map  $\bar{e}_1^{\perp} \rightarrow \bar{f}_1^{\perp}$.
Since the restriction of an isometry is an isometry, we are done. \hspace{24 mm}{$\Box$}

\section{A comparision between the algebraic and geometric approaches}

As mentioned in the introduction, our algebraic approach complements the classical geometric approach. The goal of this section is to exhibit a ``homotopy" between these two approaches.
More precisely, we will  show that our substitution in Equation (\ref{substitution})
\[x_{1} = \frac{y}{1 - m_{11}}\]
 naturally corresponds to the hyperplane reflection mentioned in the previous section.

Let us quickly recapitulate the framework:
\begin{itemize}
\item $(V, q)$ is an $n$-dimensional non-degenerate quadratic form.
\item $\{\bar{e}_1, \bar{e}_2, \dots, \bar{e}_n\}$ and  $\{ \bar{f}_1, \bar{f}_2, \dots, \bar{f}_n \}$ are two orthogonal bases for  $(V, q)$.
\item We let $q(\bar{e}_i) = a_i$  and $q(\bar{f}_i) = b_i$  for all $i$.
\item $a_1 = b_1$, i.e., $q(\bar{e}_1) = q(\bar{f}_1)$.
\item For all $i$, $w_i = w_i(x_2, x_3, \dots, x_n)$  is obtained from $z_i = z_i(x_1, \dots, x_n)$ after replacing $x_1$ with our substitution, which is a linear combination of $x_2, \dots, x_n$.
\end{itemize}

We now have two coordinate representations
\[ q_a = \la a_1, a_2, \dots, a_n \ra \ \ \text{and} \ \ \  q_b = \la b_1, b_2, \dots, b_n \ra \]
 of the  form $(V, q)$ with respect to the  bases $\{\bar{e}_i \}$ and $\{ \bar{f}_i \}$ respectively.
The isometry between $q_a$ and $q_b$ is given by an invertible matrix $M = (m_{ij})$.  The change of basis matrix is then $M$. So we have for $j = 1, \dots, n$,
\[\bar{e}_j = m_{1j} \bar{f}_1 + m_{2j} \bar{f}_2 + \cdots + m_{nj} \bar{f}_n. \]

\noindent
For  the rest of this section, we  fix an integer  $k \ge 2$.  Before going further we explain our strategy for getting the ``homotopy."  We take a  vector $\bar{e}_k$ and hit it with our hyperplane reflection $ \tau_{\bar{e}_1-\bar{f}_1}$.
Then we express $\tau_{\bar{e}_1-\bar{f}_1}(\bar{e}_k)$  as a linear combination of  $\bar{f}_2, \bar{f}_3, \dots, \bar{f}_n$. By comparing the coefficient $c_{ki}$
of $\bar{f}_i$ in  $\tau_{\bar{e}_1-\bar{f}_1}(\bar{e}_k)$ and the coefficient $d_{ki}$ of  $x_k$ in $ w_i$ for $i \ge 2$, and we will see the equivalence of the two approaches.

To execute this strategy, consider the vector $\bar{u} := \bar{e}_1 - \bar{f}_1 = ( m_{11} - 1 ) \bar{f}_1 + m_{21} \bar{f}_2 + \cdots + m_{n1} \bar{f}_n$.
Since $a_1= b_1$,  we have
\begin{eqnarray*}
q(\bar{u})  & = &   ( m_{11} - 1 )^2b_1 + m_{21}^2 b_2  + \cdots + m_{n1}^2b_n \\
            & = &   \sum_{i = 1}^n m_{i1}^2 b_i - 2m_{11}b_1 + b_1 =    b_1 - 2m_{11}b_1 + b_1=  2b_1 (1 - m_{11}).
\end{eqnarray*}
(Here we are using the identity $\sum_{i = 1}^n m_{i1}^2b_i = b_1$ which comes from unwinding the equation $B(\bar{e}_1, \bar{e}_1) = a_1 = b_1$.)
By replacing $\bar{f}_1$ with $-\bar{f}_1$ if necessary, we may assume that $m_{11} \ne 1$.
Therefore, $q(\bar{u}) \ne 0$.  Then the formula for our hyperplane reflection is given by
\[
 \tau_{\bar{u}} (\bar{z})  =   \bar{z} - \frac{2 B(\bar{z}, \bar{u})}{2 b_1(1 - m_{11})} \bar{u}
                         = \bar{z} - \frac{ B(\bar{z}, \bar{u})}{b_1(1 - m_{11})} \bar{u}.
\]

Setting $\bar{z} = \bar{e}_k$, we obtain the following equations:
\begin{eqnarray*}
 \tau_{\bar{u}} (\bar{e}_k)  & = &  \bar{e}_k - \frac{ B(\bar{e}_k, \bar{e}_1 - \bar{f}_1)}{b_1(1 - m_{11})} (\bar{e}_1 - \bar{f}_1) \\
 &  =&  \bar{e}_k + \frac{ B(\bar{e}_k, \bar{f}_1)}{b_1(1 - m_{11})} (\bar{e}_1 - \bar{f}_1) \\
& = & (m_{1k}\bar{f}_1 + \cdots + m_{nk}\bar{f}_n) + \frac{m_{1k}b_1}{b_1 ( 1 - m_{11})} ((m_{11} - 1)\bar{f}_1 + m_{21}\bar{f}_2 + \cdots + m_{n1}\bar{f}_n)\\
& = &   (m_{2k}\bar{f}_2 + \cdots + m_{nk} \bar{f}_n) + \left( \frac{m_{1k}m_{21}}{1-m_{11}}\bar{f}_2 + \cdots + \frac{m_{1k}m_{n1}}{1 - m_{11}} \bar{f}_n \right)
\end{eqnarray*}

The coefficient of $\bar{f}_i$ for $i \ge 2$  in the last expression is:
\[
\boxed{c_{ki}:= m_{ik} + \frac{m_{1k}m_{i1}}{1- m_{11}}}
\]

Now let us change gears and look at our algebraic approach.
Recall that we substitute
\[x_1 \longrightarrow \frac{y}{1 - m_{11}} \left(= \frac{m_{12}x_2 + \cdots + m_{1n}x_n}{1 - m_{11}} \right)\]
in the equations
\[ z_i = m_{i1}x_1 + \cdots + m_{in}x_n \ \ \text{ for } i = 1, 2, \dots n.\]
Using our substitution for $x_1$, for $i \ge 2$,  we get an expression for $w_i$:
\[ w_i = m_{i1}\left( \frac{m_{12}x_2 + \cdots + m_{1n}x_n}{1 - m_{11}} \right)  + m_{i2}x_2 + \cdots + m_{in}x_n.\]
The coefficient of  $x_k$   in  this  expression is given by
\[ \boxed{ d_{ki} := m_{i1} \frac{m_{1k}}{1 - m_{11}} + m_{ik}}\]
which agrees with the formula for $c_{ki}$.

In summarizing our calculations, let us show how one can see almost instantly that our substitution in Section 2 corresponds to the hyperplane reflection above. Suppose $\bar{z}$ is in the span of
$\{ \bar{e}_2, \dots, \bar{e}_n\}$. Then plugging $\bar{z}$ in the formula for $\tau_{\bar{u}}(\bar{z})$, we find that
\[ \tau_{\bar{u}}(\bar{z}) = \bar{z} + x(\bar{e}_1 - \bar{f}_1),\]
where $x$ is our substitution $x = \frac{y}{1 - m_{11}}$.
But when one reflects on the corresponding map (related to our substitution)
\[ \Phi  \colon \bar{e}_1^{\perp} \rar \bar{f}_1^{\perp},\]
one sees that
\[ \Phi(\bar{z}) = \bar{z} + x \bar{e}_1 - t \bar{f}_1, \]
where $t$ is a uniquely determined element of $F$  such that the projection of $\Phi(\bar{z})$ on the line through $\bar{f}_1$ is $0$.  Since  our image of reflection $\tau_{\bar{u}}(\bar{z})$ already has this property, we see that $x = t$ and $\tau_{\bar{u}}(\bar{z}) = \Phi(\bar{z})$.

In conclusion,  we have seen that our substitution
\[ x_1 \longrightarrow \frac{y}{1 - m_{11}}\]
amounts to reflecting vectors in the plane orthogonal to the vector $\bar{u}$, i.e., sending  $\bar{z}$ to $\tau_{\bar{u}} (\bar{z})$. Thus, we have established a
``homotopy" between the algebraic and geometric approaches.

\section{What is the Witt ring of quadratic forms?}

In this section we will define the Witt ring of quadratic forms.
As we will see, the Witt cancellation theorem will be the key for constructing the Witt ring.  Some terminology is in order. We refer the reader to
the excellent books by Lam \cite{Lam, Lam2} for a thorough treatment.  Other good references on this subject include
\cite{ekm, larry, Pfi, Scharlau, Szy}.

Let $(V, B)$ be a quadratic space and let $q$ be the corresponding quadratic form.  For simplicity we often drop $B$ and $q$ and denote a quadratic space by $V$.
Recall that a quadratic space
$(V, B)$ is said to be \emph{non-degenerate}  if the induced map
\[ B(\bar{v}, -) \colon V \rar F\]
is the zero map only when $\bar{v} = 0$. It is not hard to show that any quadratic space $(V, B)$ splits as
\[ V = V_{non-deg} \,\bot \,  V_{null}\]
where $V_{non-deg}$ is non-degenerate, and $V_{null}$ is  the subspace of $V$ consisting of all vectors in $V$ which are orthogonal to all vectors of $V$.
In particular, the restriction of the bilinear form $B$ on $V_{null}$ is identically $0$.
Therefore there is no harm in restricting to non-degenerate quadratic spaces.

We say that a non-degenerate quadratic space $(V, B)$  is  \emph{isotropic} if there is a non-zero vector $\bar{v}$ such that $q(\bar{v}) = 0$.  It can be shown  \cite{Witt} that every isotropic form contains a hyperbolic plane as a summand, where, by definition, a hyperbolic plane is a  two dimensional form that is equivalent to $\la 1, -1 \ra$. Note that $\la 1, -1\ra$ is short for
$x_1^2 - x_2^2$. This form $q$ is isotropic as $q(1, 1) = 0$. Thus we see that a non-degenerate quadratic form $V$ is isotropic if and only if $V$ has a hyperbolic plane as a summand.

Now let us consider a non-degenerate quadratic space $(V, B)$. If $V$ is isotropic,  then by the above mentioned fact we can write $V$ as
\[ V = H_1 \, \bot \, V_1 ,\]
where $H_1$ is a hyperbolic plane.  If $V_1$ is also isotropic, we can further decompose it as
\[ V = H_1 \, \bot \, (H_2 \, \bot \, V_2),\]
where  $H_2$ is a hyperbolic plane.  We proceed in this manner as far as possible, to get a decomposition:
\[ V = H_1 \, \bot \, H_2  \, \bot \, \dots \, \bot \, H_k \, \bot \, V_a\]
where $H_i$ are hyperbolic planes and $V_a$ is anisotropic, i.e.,  a form that is not isotropic. Now here is where  Witt cancellation comes into play. The integer $k$ (the number of hyperbolic  planes in the above decomposition) is seen to be uniquely determined, using the Witt cancellation theorem. Furthermore, the isometry class of the anisotropic part $V_a$ is uniquely determined, which also follows from the Witt cancellation theorem.
In summary, every non-degenerate quadratic space $(V, B)$ admits a unique decomposition called the \emph{Witt decomposition}
\[ V = H \, \bot \, V_a,\]
where  $H$ is a sum of hyperbolic planes and $V_a$ is anisotropic.
Two quadratic spaces $V$ and $W$ are said to be \emph{similar} if their anisotropic parts are equivalent.  Again, the Witt cancellation theorem ensures that this notion of similarity is well-defined.


With these definitions and concepts, we are now ready to define the Witt ring of quadratic forms $W(F)$ over the field $F$, which is a central object in the algebraic theory of quadratic forms.
The elements of $W(F)$ are the similarity classes of quadratic forms. Since these classes are uniquely represented up to equivalence by anisotropic quadratic forms, we can think of  $W(F)$ as the set of equivalence classes of anisotropic quadratic forms.
 Given two such elements $(V, B_V)$ and $(W, B_W)$, the ring operations  of addition and multiplication are defined by
\begin{eqnarray*}
V + W & := & (V \, \bot \, W)_a,  \ \ \text{ and }\\
V  W & := &  (V \otimes W)_a.
\end{eqnarray*}
Our tensor space $V \otimes W$
\footnote{Our reader  can think about tensor products as a  target of some kind of  ``universal bilinear form" which one can define precisely.  Each element in $V \otimes W$ is a sum $v_1 \otimes w_1 + \cdots + v_k \otimes w_k$ where $k$ is in $\mathbb{N}$, and $v_i \otimes w_i$ is in the image of this bilinear form. Also, the dimension of the tensor product is a product of the dimensions of $V$ and $W$.  For a nice introduction to tensor products see \cite[Chapter 10, Section 4]{DumFoo}.}  is equipped with a bilinear form $B$ defined by
\[ B(v_1 \otimes w_1, v_2 \otimes w_2) = B_V(v_1, v_2) B_W(w_1, w_2). \]
  These operations give $W(F)$ the structure of a commutative ring.
The zero quadratic space  is vacuously anisotropic and is the additive identity for $W(F)$, and the one dimensional form $\la 1 \ra$ is the multiplicative identity for $W(F)$.
Further details and proofs can be found in \cite[Chaper 2, Section 1]{Lam}.

 Even though these ideas were all present in Witt's paper \cite{Witt} from 1937, the algebraic theory of quadratic forms had many years of slow growth before receiving a considerable spark from the remarkable work of Pfister \cite{Pf2, Pf3} in the 1960's.   In particular,  Pfister's work generated intense interest in powers of the so-called \emph{fundamental ideal}, $I(F)$, defined in the next section.

\section{Milnor and Bloch-Kato conjectures}
Milnor, in his celebrated paper \cite{Mil} indicated a close and deep connection between three central arithmetic objects:  an associated graded ring of the Witt ring $W(F)$ of quadratic forms,  the Galois cohomology ring $H^*(F, \mathbb{F}_2)$ of the absolute Galois group, and the  reduced Milnor $K$-theory ring $K_{*}(F)/2$. In this section we will touch on these topics very briefly to show the reader the connection between the Witt ring and these topics. The interested reader is encouraged to see \cite{Mil, NSW} for more details.  The connection between the Witt ring
and Galois theory is investigated in \cite{MinSpi}.

\subsection{Associated graded Witt ring}
Let $I(F)$, or simply $I$, denote the ideal of $W(F)$  consisting of elements which are represented by even dimensional anisotropic quadratic forms. As an additive subgroup of $W(F)$ this is generated by forms $\la 1, a \ra$, and therefore $I^{n}$  is additively generated by the so-called \emph{$n$-fold Pfister  forms} $\la 1, a_{1}\ra \la 1, a_{2} \ra  \dots \la 1, a_{n}\ra$ in the Witt ring; see \cite[Page 36]{Lam2}. By convention
$I^0 = W(F)$. The associated graded Witt ring is then
\[ \bigoplus_{n \ge 0} \frac{I^{n}}{I^{n+1}} = \frac{W(F)}{I} \oplus \frac{I}{I^{2}} \oplus \frac{I^{2}}{I^{3}} \oplus \dots. \]
The three classical invariants of quadratic forms, namely dimension $e_0$, discriminant $e_1$, and Clifford invariant $e_2$, are defined as homomorphisms on the first three summands respectively as follows:
\begin{align*}
e_0 \colon \frac{W(F)}{I} \longrightarrow \mathbb{F}_2,  & \ \ \ e_0([q]) = \dim q \,(\text{mod} 2). \\
e_1 \colon \frac{I}{I^2} \longrightarrow \frac{F^*}{(F^*)^2}, &  \ \ \ e_1([q]) = [(-1)^{\frac{n(n-1)}{2}} \det q] , \ \ \text{where } n = \dim q.\\
 e_2 \colon \frac{I^2}{I^3} \longrightarrow  B(F), & \ \ \   B(F) \text{ stands for the Brauer group of }  F.
\end{align*}
The Brauer group $B(F)$ of a field $F$ can be informally thought as the set of all 
division algebras over $F$ endowed with an interesting product giving it the structure of a group. For a precise definition see \cite[Chapter 4, Section 1]{Lam} or \cite[Chapter 12, Section 12.5]{pierce}.  \footnote{It is worthwhile to observe that there 
 is an analogy between the construction of the Witt ring of quadratic forms and Brauer groups. In this analogy quadratic forms $q$ over $F$ are analogous to the group $M_n(D)$ of $n \times n$ matrices over a division algebra $D$ over $F$. 
 The anisotropic quadratic forms are analogous to division algebras, and the Witt 
 cancellation theorem plays a similar role to  Wedderburn's theorem which implies that 
 $M_n(D)$ uniquely determines $D$ up to isomorphism. }

 Quadratic forms would be completely
classified by these classical invariants if $I^{3} = 0$; see \cite[Page 374]{elmlam}.  However, that is not true in general. So one has to look for higher invariants. Milnor was able to do this by extending these classical invariants into an infinite family of invariants, taking values in the Galois cohomology ring of $F$. This brings us to the next object of interest.

\subsection{Galois cohomology}
Let  $\overline{F}_ {sep}$ denote the separable closure of a field $F$ with characteristic not equal to $2$. One of the main goals of algebraic number theory and arithmetic geometry is to understand the structure of the absolute Galois group $G_{F} = \Gal(\overline{F}_ {sep} / F)$.  To understand this group better one associates a cohomology theory to this group called \emph{Galois cohomology}, which is a graded object. (In this section we use only $\mathbb{F}_{2}$ coefficients because we are interested in connections of Galois cohomology  with quadratic forms. However other coefficients also play an important role in Galois cohomology. See our remarks on the Bloch-Kato conjecture at the end of the next section.)
\[ H^{*}(F, \mathbb{F}_{2}) = H^{0}(F, \mathbb{F}_{2}) \oplus H^{1}(F, \mathbb{F}_{2}) \oplus \cdots  \]
The first two groups are easy to define.  $H^{0}(F, \mathbb{F}_{2}) = \mathbb{F}_{2}$, and $H^{1}(F, \mathbb{F}_{2})$ is the group of continuous homomorphism from $G_{F}$ to $\mathbb{F}_{2}$.
See \cite{NSW, Serre} for the general definition.   $H^{*}(F, \mathbb{F}_{2})$ is also equipped with the structure of commutative ring.

 For certain fields $F$, Milnor proved \cite{Mil} the existence of a well-defined map
 \[ e \colon \oplus_{n \ge 0} I^{n}/I^{n+1} \rightarrow  \oplus_{n \ge 0} H^{n}(F, \mathbb{F}_{2}),\]
and he showed that it is an isomorphism.  In \cite{Mil} he asked if the same is true in general.  For an arbitrary $F$, even showing that $e$ is a well-defined map is very hard.
This problem, of showing that $e$ is a well-defined map and that it is an isomorphism for all $F$, is known as the \emph{Milnor conjecture on quadratic forms}.
This problem has fascinated mathematicians and was eventually settled affirmatively in  \cite{OVV}.

\subsection{Reduced Milnor $K$-theory}
The ring structure on both  the domain and the target  of the map $e$  is mysterious. To explain this ring structure
Milnor constructed a  third object, now called reduced Milnor $K$-theory $K_{*}(F)/2$, whose ring structure is far more  transparent.  Let $F^*$  be the multiplicative group of non-zero elements in $F$.  The tensor algebra $T(F^*)$ is a graded algebra defined by
\[ T(F^*) := \mathbb{Z} \bigoplus F^* \bigoplus (F^* \otimes F^*) \bigoplus (F^* \otimes F^* \otimes F^*) \bigoplus \cdots.\]
The reduced Milnor $K$-theory  $K_*(F)/2$ is  the tensor algebra $T(F^*)$ modulo the two-sided ideal  $\la a \otimes b \, | \, a + b = 1,\ \  a, b  \in F^{*} \ra$ reduced modulo $2$. That is,
\[ K_{*}(F)/2  :=  \frac{T(F^{*})}{ \la a \otimes b \, | \, a + b = 1, \ \  a, b  \in F^{*} \ra} \otimes \mathbb{F}_{2}\]Milnor defined  two families of maps $\nu$ and  $\eta$ shown in the triangle below. (The map $\eta$ was defined using a lemma of Bass and Tate \cite[Lemma 6.1]{Mil}.)  Showing that all  maps in this triangle are isomorphisms  was a major problem in the field and it  went under the name of  \emph{ The Milnor conjectures}. The map $\eta$ was shown to be an isomorphism by Voevodsky,  for which he won the Fields medal in 2002.  As mentioned earlier, $e$ was shown to be an isomorphism in \cite{OVV}, building upon the work of Voevodsky. These theorems are among the most powerful results in the algebraic theory of quadratic forms.  For further details and proofs of these theorems see \cite{Mer, OVV, Wei}.

The  Milnor triangle is the triangle connecting quadratic forms, Galois cohomology and the reduced Milnor $K$-theory:
\[
\xymatrix{
 &  K_*(F)/2\ar[dl]_{\nu} \ar[dr]^{\eta} & \\
 \oplus_{n \ge 0} I^{n}/I^{n+1} \ar[rr]_e & &  \oplus_{n \ge 0} H^{n}(F, \mathbb{F}_{2})
}
\]

For odd primes $p$, a similar isomorphism was conjectured by Bloch and Kato, between the reduced Milnor $K$-theory $K_*(F)/p$ and the Galois cohomology ring $H^*(F, \mathbb{F}_p)$ when the field $F$ contains a primitive $p$-th root of unity. This  \emph{Bloch-Kato conjecture} was proved  in 2010 by Rost and Voevodsky, with a contribution from Weibel.
The interested reader can consult \cite{Rost, Wei, Voe}. The background required for these deep, very recent papers is quite extensive,
so the ambitious reader will no doubt have lots of fun delving into many  \emph{extra} references, including those found in the references of the papers we cite.

\section{The Dickson-Scharlau surprise}

 After essentially  completing our article we kept searching for historical references on quadratic forms. We were astounded to find a conference proceeding article \cite{Scharlau2} by W. Scharlau entitled, ``\emph{On the history of the algebraic theory of quadratic forms.}" Scharlau explains that the algebraic theory of quadratic forms could have been born 30 years earlier!
Namely, in 1907 L. Dickson  published a paper \cite{Dick} in which he proved a number of results on quadratic forms including the cancellation theorem which Witt proved independently 30 years later in 1937.   In fact,  Scharlau writes: ``...\emph{ It seems that Dickson's paper went completely unnoticed; I could not find a single reference to it in the literature. However, one must admit that this paper -- like most of Dickson's work -- is not very pleasant to read... Nevertheless, I believe that, as far as Witt's theorem and related questions are concerned, some credit should be given to Dickson.}" Therefore,  one might say that the algebraic theory of quadratic forms was \emph{conceived} in 1907, but wasn't \emph{born} until 1937.

\vskip 5mm
\noindent
\emph{Acknowledgements:}  We  would like to thank David Eisenbud, Edward Frenkel, Margaret Jane Kidnie, John Labute, Claude Levesque,  Alexander Merkurjev, John Milnor, Raman Parimala, Andrew Ranicki, Balasubramanian Sury, Stefan Tohaneanu,  Ravi Vakil and Charles Weibel for their encouragement, help with the exposition, and nice  welcome of the preliminary version of our paper. Last but not  the least we thank Matthew Teigen for his nice illustration of the hyperplane reflection.

\end{document}